\documentclass[11pt,leqno]{amsart}

\makeatletter
\renewcommand\subsection{\@startsection{subsection}{2}%
  \z@{.5\linespacing\@plus.7\linespacing}{-.5em}%
  {\normalfont\scshape}}
\makeatother
\usepackage[a4paper,lmargin=2cm,rmargin=2cm,tmargin=4cm,bmargin=4cm]{geometry}

\usepackage[centertags]{amsmath}
\usepackage{amsfonts}
\usepackage{amssymb}
\usepackage{amsthm}
\usepackage{dsfont}
\usepackage{mathabx}

\numberwithin{equation}{section}

\usepackage[pdftex]{color}
\usepackage[bookmarks=true,hyperindex,pdftex,
colorlinks, citecolor=blue,linkcolor=blue, urlcolor=blue
]{hyperref}

\usepackage{tikz}
\usepackage{tikz-cd}
\usetikzlibrary{decorations.markings}

\usepackage[normalem]{ulem}
\usepackage{bbm}

\renewcommand{\geq}{\geqslant}
\renewcommand{\leq}{\leqslant}

\theoremstyle{plain}
\newtheorem{theorem}{Theorem}[section]
\newtheorem{corollary}[theorem]{Corollary}
\newtheorem{lemma}[theorem]{Lemma}
\newtheorem{proposition}[theorem]{Proposition}

\theoremstyle{definition}

\newtheorem{remark}[theorem]{Remark}

\newtheorem{example}[theorem]{Example}

\newtheorem{definition}[theorem]{Definition}

\author{Andr\'es Quilis}

\address[Quilis] {Universitat Polit\`ecnica de Val\`encia. Instituto Universitario de Matem\'atica Pura y Aplicada, Camino de Vera, s/n 46022 Valencia, Spain; and Czech Technical University in Prague, FEE, Dept.of mathematics, Technicka 2, 160 00 Prague 6, Czech Republic
 \newline
\href{https://orcid.org/0000-0001-6022-9286}{ORCID: \texttt{0000-0001-6022-9286} } }
\email{\texttt{anquisan@posgrado.upv.es}}

\author{Abraham Rueda Zoca}

\address[Rueda Zoca] {Universidad de Granada, Facultad de Ciencias.
Departamento de An\'{a}lisis Matem\'{a}tico, 18071 Granada, Spain
 \newline
\href{https://orcid.org/0000-0003-0718-1353}{ORCID: \texttt{0000-0003-0718-1353} } }
\email{\texttt{abrahamrueda@ugr.es}}

\title{(Almost isometric) local retracts in metric spaces}

\begin{document}



\maketitle

\begin{abstract}
We introduce the notion of (almost isometric) local retracts in metric space as a natural non-linear version of the concepts of locally complemented and almost isometric ideals from Banach spaces. We prove that given two metric spaces $N\subseteq M$ there always exists an almost isometric local retract $S\subseteq M$ with $N\subseteq S$ and $dens(N)=dens(S)$. We also prove that metric spaces which are local retracts (respectively almost isometric local retracts) can be characterised in terms of a condition of extendability of Lipschitz functions (respectively almost isometries) between finite metric spaces. Different examples and counterexamples are exhibited.
\end{abstract}

\section{Introduction}

The study of complemented subspaces has attracted a lot of attention in Banach space theory because the structure of complemented subspaces of a given Banach space may reveal a lot of information about its geometry and about other structural properties. Probably one of the most famous results in this area is the well known theorem of Lindenstrauss and Tzafriri which asserts that, given a Banach space $X$, every closed subspace of $X$ is complemented if, and only if, $X$ is isomorphic to a Hilbert space \cite{litzacom}. This theorem reveals that the concept of complemented subspace, though highly useful and interesting, may be too strong for a general Banach space (the most pathological examples in this regard are the so called \textit{hereditarily indecomposable Banach spaces}, in which every infinite-dimensional closed subspace fails to have non-trivial complemented subspaces, see \cite[Section 5.4]{fab2}).

Weaker notions than linear complementation, still having strong consequences on the given subspaces, have appeared in the literature in connection with local theory of Banach spaces. These are the concepts of \textit{locally complemented subspaces} and \textit{almost isometric ideals} in Banach spaces. Given a Banach space $X$ and a subspace $Y$ of $X$, we say that $Y$ is an \textit{almost isometric ideal} in $X$ if, given any finite-dimensional subspace $E$ of $X$ and any $\varepsilon>0$, there exists a linear operator $T:E\rightarrow Y$ such that
\begin{enumerate}
    \item $T(e)=e$ holds for $e\in E\cap Y$ and;
    \item $(1-\varepsilon)\Vert x\Vert\leq \Vert T(x)\Vert\leq (1+\varepsilon)\Vert x\Vert$ holds for every $x\in E$.
\end{enumerate}
If we require in (2) only $\Vert T(x)\Vert\leq (1+\varepsilon)\Vert x\Vert$ then we say that $Y$ is \textit{locally complemented in $X$}. 

The notion of locally complemented subspaces can be found for instance in \cite{kalton84} (see also \cite{fak72}). On the other hand, the notion of almost isometric ideal appeared in \cite{aln2}. 

The advantage of considering locally complemented subspaces and almost isometric ideals is that, even though both notions still have strong implications on the given subspaces (see e.g. \cite[Theorem 3.5]{kalton84} for a connection between local complementability and the extendability of compact operators; or \cite[Theorem 4.2]{aln2} for a description of Gurarii spaces in terms of almost isometric ideals), these notions are quite abundant in every Banach space. 

This fact was first crystallised by Heinrich and Mankiewicz in \cite[Proposition 3.4]{HM82}, who proved the following result using model-theoretic tools: Given any Banach space $X$ and any subspace $Y$ of $X$ there exists a subspace $Z$ of $X$ containing $Y$ with $dens(Z)=dens(Y)$ and such that $Z$ is locally complemented in $X$. Later, Sims and Yost in \cite[Theorem]{sy89} offered a proof using a geometric lemma by Lindenstrauss. In \cite[Theorem 1.4]{Abrahamsen15}, Abrahamsen extended this result by replacing ``locally complemented'' with ``almost isometric ideal''.

In some contexts, the above-mentioned result can be improved in order to give more information regarding the properties of the locally complemented subspace $Z$. This is the case of the \textit{Lipschitz-free spaces} setting (see formal definition in Section 2) in connection with \cite[Theorem 5.3]{HajQui22}, where the following theorem is proved: Given any metric space $M$ and any subspace $N$ of $M$ there exists a subspace $S$ of $M$ containing $N$ with $dens(S)=dens(N)$ and such that $\mathcal F(S)$ is locally complemented in $\mathcal F(M)$. In other words, \cite[Theorem 5.3]{HajQui22} says that, in Heinrich's and Mankiewicz' theorem, if $X$ is a Lipschitz-free space $\mathcal F(M)$ and $Y$ is a Lipschitz-free space $\mathcal F(N)$ for some $N\subseteq M$, then $Z$ can be found of the form $\mathcal F(S)$ for $N\subseteq S\subseteq M$.

This motivates the question of whether the local complementation in \cite[Theorem 5.3]{HajQui22} may come from any version of ``local retracts'' in the underlying metric spaces. Motivated by this question, the aim of this paper is to introduce the notions of local retracts and almost isometric local retracts in metric spaces and prove a metric version of \cite[Theorem 1.4]{Abrahamsen15}. To be more precise, we introduce the notions of local retracts and almost isometric local retracts in Definition \ref{defi:localretracts}, showing that they are different properties. We prove in Theorem \ref{Theorem:LocalRetract_ExtensionintoProper} that when $N$ is a local retract in $M$ then, for every proper metric space $M$ and every Lipschitz function $f:N\longrightarrow X$ there exists a norm-preserving extension $F:M\longrightarrow X$. This allows us to conclude that, when $X$ is a Banach space, a closed linear subspace $Y$ is a local retract in $x$ if, and only if, $Y$ is locally complemented in $X$, which establishes a natural link between there two notions in the framework of Banach spaces. 

The above-mentioned Theorem \ref{Theorem:LocalRetract_ExtensionintoProper} also motivates us to characterise the metric spaces $M$ which are a local retract in any metric space containing it. In the language of Definition \ref{defi:absolutelocal} we prove in Theorem \ref{theo:carahyperconvexelocalretracevery} that a metric space $M$ is an absolute local retract if, and only if, $M$ extends in a norm-preserving way any Lipschitz function between finite domains (i.e. $M$ is $1$-finitely injective). In the context of absolute almost isometric local retracts we prove that a metric space $M$ is an almost isometric local retract in $M$ if, and only if, $M$ extends $(1+\varepsilon)$-almost isometries between any pair of finite metric spaces. In contrast to what happens in the case of absolute local retracts, where every isometric $L_1$-predual is an absolute local retract, the only separable metric space $M$ which is an absolute almost isometric local retract is the Urysohn space $\mathbb U$ (see Remark \ref{remark:polishurysohn}).

We devote the last section of the paper to prove Theorem \ref{theo:metricsimyostnoseparable}. This theorem establishes that, given any metric space $M$ and any $N\subseteq M$ there exists an almost isometric local retract in the middle $N\subseteq S\subseteq M$ with $dens(S)=dens(N)$, which improves \cite[Theorem 5.3]{HajQui22}.

\section{Notation and preliminary results}

All Banach spaces considered in this article will be real. Given a Banach space $X$, we will use $B_X$ to denote its closed unit ball, and $X^*$ to denote its topological dual. Given a Banach space $X$, a subspace $Y$ of $X$ and a constant $\lambda\geq 1$, we say that $Y$ is $\lambda$-complemented in $X$ if there exists a linear projection $P$ from $X$ onto $Y$ whose norm is bounded by $\lambda$. 

In a metric space $M$, the closed ball centered at a point $x\in M$ of radius $\delta$ will be denoted by $B(x,\delta)$. The density character of a metric space $M$ will be denoted by $\text{dens}(M)$.

The Lipschitz constant of a map $F$ between metric spaces will be denoted by $\|F\|_\text{Lip}$. Given a constant $\lambda\geq 0$, a map $F$ between metric spaces is said to be $\lambda$-Lipschitz if $\|F\|_\text{Lip}\leq \lambda$.

In analogy with the linear setting, given a metric space $M$, a subset $N$ and a constant $\lambda\geq 1$, we say that $N$ is a $\lambda$-Lipschitz retract of $M$ if there exists a $\lambda$-Lipschitz $R$ from $M$ onto $N$. We say that $N$ is a Lipschitz retract of $M$ if it is a $\lambda$-Lipschitz retract for some $\lambda\geq 1$. 

We will use the following characterisation of local complementability, whose proof can be obtained from \cite{kalton84} and \cite{fak72}.
\begin{theorem}
\label{Theorem:Local_Comp_Eq}
Let $X$ be a Banach space, $Y\subset X$ a linear subspace. The following statements are equivalent:
\begin{itemize}
    \item[(1)] $Y$ is locally complemented in $X$.
    \item[(2)] There exists a linear projection $P\colon X^*\rightarrow Y^{\perp}$ such that $\|\text{Id}_{X^*}-P\|\leq 1$.
    \item[(3)] $Y^{**}$ is $1$-complemented in $X^{**}$ in its natural embedding.
    \item[(4)] $Y$ has the Compact Extension Property in $X$, i.e.: for every Banach space $Z$ and every linear compact operator $K\colon Y\rightarrow Z$, there exists a compact operator $\widehat{K}\colon X\rightarrow Z$ that extends $K$ and such that $\|\widehat{K}\|\leq \|K\|$. 
    \item[(5)] There exists a linear extension operator $E\colon Y^*\rightarrow X^*$ with $\|E\|\leq 1$.
    \item[(6)] There exists a linear extension operator $E\colon \text{Lip}_0(Y)\rightarrow \text{Lip}_0(X)$ with $\|E\|\leq 1$.
\end{itemize}
\end{theorem}

Note that the concept of local complement can be parameterized with a constant $\lambda$ controlling the norm of the linear operator used in the definition. In our case, we chose not to parameterize the definition of neither local complements nor local retractions, since we are interested in the local isometric structure. 

Given a metric space $M$ with a distinguished point $0\in M$, we write $\text{Lip}_0(M)=\{f\colon M\rightarrow \mathbb{R}\colon f\text{ is Lipschitz and } f(0)=0\}$, which is a Banach space when endowed with the norm given by the best Lipschitz constant. Its canonical predual is the Lipschitz-free space, which is given by $\mathcal{F}(M)=\overline{span}\{\delta(x)\colon x\in M\}\subset \text{Lip}_0(M)^{**}$. We will use its structure and its fundamental properties, for which we refer the reader to, for instance, \cite{GK03} or the monograph \cite{Wea18} (where this space receives the name of Arens-Ells space).

Some more definitions will be recalled throughout the article. 

\section{Definitions and properties}

Let us start with the following definition.

\begin{definition}\label{defi:localretracts}
Let $M$ be a metric space and let $N$ be a subset of $M$.

\begin{enumerate}
    \item We say that $N$ is \textit{a local retract} of $M$ if, for every finite subset $E$ of $M$ and every $\varepsilon>0$, there exists a $(1+\varepsilon)$-Lipschitz map $r:E\rightarrow N$ such that $r(e)=e$ holds for every $e\in E\cap N$.
    \item We say that $N$ is \textit{an almost isometric local retract}  (\textit{ai-local retract} for short) of $M$ if, for every finite subset $E$ of $M$ and every $\varepsilon>0$, there exists a map $r:E\rightarrow N$ such that $r(e)=e$ holds for every $e\in E\cap N$ and that
    $$(1-\varepsilon)d(x,y)\leq d(r(x),r(y))\leq (1+\varepsilon)d(x,y)\ \forall x,y\in E.$$
\end{enumerate}
\end{definition}

It is clear from the definitions that any ai-local retract is automatically a local retract. It is also straightforward to check that any $1$-Lipschitz retract is a local retract. However, the concepts of Lipschitz retracts and ai-local retracts are in general unrelated; i.e.: neither concept implies the other.

We start with an example which shows that $1$-Lipschitz retracts (and in particular local retracts) may fail to be ai-local retracts.

\begin{example}
Fix $n\in\mathbb N$, and set $N:=\{1,\ldots, n\}$ and $M:=\mathbb N$ with the usual distance inherited from the real line. The mapping $r:M\rightarrow N$ by
$$r(x):=\left\{\begin{array}{cc}
    r(x) & \mbox{if }x\in N, \\
    n & \mbox{if }x\geq n,
\end{array} \right.$$
is a $1$-Lipschitz retract. However, it is immediate that $N$ is not an ai-local retract in $M$ because, by the very definition, any ai-local retract in an infinite metric space must be infinite.
\end{example}

Examples of metric spaces which are ai-local retracts and not Lipschitz retracts are considerably less elementary. The first of such examples is derived from the work of Kalton, who, in \cite{kalton11} solved an open problem of Lindenstrauss (see \cite{Lin66}) by constructing a non-separable Banach space which is not a Lipschitz retract of its bidual. In order to use Kalton's result for our purposes, we first discuss the relationship between the newly introduced notions of local retract and ai-local retract, and the linear ideas of local complements and ai-ideals:

\begin{remark}
    \label{Remark:LocCompl_implies_LocRetr_inBanach}
    In Banach spaces, the concept of local retract is naturally weaker than local complement. Indeed, let $X$ be a Banach space and let $Y$ be a subspace of $X$ which is locally complemented. Given a finite set $E$ and $\varepsilon>0$, we can find a linear map $T\colon \text{span}(E)\rightarrow Y$ with $\|T\|\leq 1+\varepsilon$ and such that $Te=e$ for all $e\in\text{span}(E)$. Restricting $T$ to the spanning set $E$, we obtain a $(1+\varepsilon)$-Lipschitz map $r=T_{|E}\colon E\rightarrow Y$ which fixes every point in $E\cap Y$. Therefore, $Y$ is a local retract of $X$. Analogously, we have that if $Y$ is an ai-ideal of $X$, then $Y$ is also an ai-local retract of $X$.

    As mentioned above, this allows us to give an example of an ai-local retract which fails to be a Lipschitz retract, in the category of Banach spaces. Indeed, Kalton constructed in \cite{kalton11} a non-separable Banach space $X$ which is not a Lipschitz retract of its bidual $X^{**}$. Since every Banach space is an ai-ideal in its bidual by virtue of the Local Reflexivity Principle, we get that $X$ is an ai-local retract in $X^{**}$ which fails to be a Lipschitz retract. 

    By the end of this article (Remark \ref{Remark:Skein}), we will be able to show that there exists a separable metric space $N$ which is an ai-local retract of a non-separable metric space $M$, while failing to be a Lipschitz retract. 
\end{remark}

We focus now on the concept of local retracts, for which we can elaborate further with respect to the previous remark. Indeed, we are going to see that the concepts of local retract and local complement coincide in the linear setting of Banach spaces. First, we need the following general theorem, which should be compared to Kalton's characterisation of local complementability through extension of linear compact operators (see (4) in Theorem \ref{Theorem:Local_Comp_Eq}), proven in \cite{kalton84}.

\begin{theorem}
\label{Theorem:LocalRetract_ExtensionintoProper}
    Let $M$ be a metric space and $N$ be a local retract in $M$. Then, if $X$ is a proper metric space, for every Lipschitz function $f:N\rightarrow X$ there exists an extension $F:M\rightarrow X$ such that $\Vert F\Vert_\text{Lip}=\Vert f\Vert_\text{Lip}$.

    Moreover, in the particular case when $X=\mathbb{R}^n$ for any $n\in\mathbb{N}$, there exists a linear extension operator $T\colon \text{Lip}_0(N,\mathbb{R}^n)\rightarrow\text{Lip}_0(M,\mathbb{R}^n)$ such that $\|Tf\|_\text{Lip}=\|f\|_\text{Lip}$ for all $f\in\text{Lip}_0(N,\mathbb{R}^n)$. 
\end{theorem}

\begin{proof}

We will prove the result when $X=\mathbb{R}^n$ for a fixed $n\in\mathbb{R}^n$, since the additional claim of the existence of a linear extension operator requires a slightly more technical approach in one step of the proof. We will point out in the aforementioned step the strategy to show the simpler statement for general proper metric spaces.

We will perform a classical Lindentrauss compactness argument. In order to do so, call $\Gamma:=\{(E,\varepsilon): E\subseteq M\mbox{ is finite and } 0<\varepsilon\leq 1\}$. We endow $\Gamma$ with the partial order $\leq$ given by $(E,\varepsilon)\leq (F,\delta)$ if and only if $E\subseteq F$ and $\delta\leq \varepsilon$. With this order $\Gamma $ is a directed set.

Given $(E,\varepsilon)\in \Gamma$, since $N$ is a local retract in $M$, there exists a $(1+\varepsilon)$-Lipschitz mapping $r_{(E,\varepsilon)}:E\rightarrow N$ satisfying that $r_{(E,\varepsilon)}(e)=e$ for every $e\in E\cap N$. Using this map, we can define, for every Lipschitz function $f\in\text{Lip}_0(N,\mathbb{R}^n)$, a (non-Lipschitz) map $\widehat{f}_{(E,\varepsilon)}\colon M\rightarrow \mathbb{R}^n$ given by:
$$\widehat{f}_{(E,\varepsilon)}(x):=\left\{\begin{array}{cc}
 f(r_{(E,\varepsilon)}(x))    & \mbox{ if }x\in E,  \\
    0 & \mbox{otherwise}. 
\end{array} \right.$$
Note that, for a fixed $(E,\varepsilon)\in\Gamma$, the point  $\widehat{f}_{(E,\varepsilon)}(x)$ belongs to the compact set $B(0,(1+\varepsilon)\|f\|_\text{Lip}d(x,0))$ for every $f\in\text{Lip}_0(N,\mathbb{R}^n)$ and every $x\in M$. Therefore, the set 
$$ \{(\widehat{f}_{(E,\varepsilon)}(x))_{f\in \text{Lip}_0(N,\mathbb{R}^n),x\in M}\colon (E,\varepsilon)\in \Gamma\},$$
indexed by the directed set $\Gamma$, is a net in the product space $\prod_{\substack{f\in\text{Lip}_0(N,\mathbb{R}^n)\\ x\in M}}B(0,2\|f\|_\text{Lip}d(x,0))$, which is compact by Tychonoff's Theorem.

In the general case, when we only need to extend a single function $f$ with image in a general proper metric space $X$, we would simply consider the product space $\prod_{x\in M}B(0,\|f\|_\text{Lip}d(x,0))$, which is likewise a compact topological space. The rest of the proof proceeds in the same way for both cases.

Using compactness, consider a cluster point 
$$F= (F_{(f,x)})_{\substack{f\in\text{Lip}_0(N,\mathbb{R}^n)\\ x\in M}}\in \prod_{\substack{f\in\text{Lip}_0(N,\mathbb{R}^n)\\ x\in M}}B(0,2\|f\|_\text{Lip}d(x,0))$$
of the previously defined net, and define the map $T\colon \text{Lip}_0(N,\mathbb{R}^n)\rightarrow \text{Lip}_0(M,\mathbb{R}^n)$ by $(Tf)(x)=F_{(f,x)}$ for every $f\in\text{Lip}_0(N,\mathbb{R}^n)$ and $x\in M$. We will show that $T$ is a well defined linear extension operator with $\|T\|=1$. 

We start by showing that $\|Tf\|_\text{Lip}\leq \|f\|_\text{Lip}$ for every $f\in\text{Lip}_0(N,\mathbb{R}^n)$. Fix such a function $f$ and two points $x,y\in M$. Given any $0<\delta\leq1$, by definition of cluster point, and since the product topology is the topology of pointwise convergence, we have that the set 
$$ A_\delta=\left\{(E,\varepsilon)\in \Gamma\colon d\left(F(f,x),\widehat{f}_{(E,\varepsilon)}(x)\right)<\delta\text{ and }d\left(F(f,y),\widehat{f}_{(E,\varepsilon)}(y)\right)<\delta\right\}$$
is cofinal in $\Gamma$. Hence, given $\left(\{x,y\},\delta\right)\in \Gamma$, there exists $(E,\varepsilon)\in A_\delta$ such that $\{x,y\}\subset E$ and $\varepsilon\leq \delta$. By definition of $\widehat{f}_{(E,\varepsilon)}(x)$, we obtain that
\begin{align*}
    d\left((Tf)(x),(Tf)(y)\right)&\leq d\left(F(f,x),\widehat{f}_{(E,\varepsilon)}(x)\right)+d\left(\widehat{f}_{(E,\varepsilon)}(x),\widehat{f}_{(E,\varepsilon)}(y)\right)+d\left(\widehat{f}_{(E,\varepsilon)}(y),F(f,y)\right)\\
    &<2\delta + d\left(f(r_{(E,\varepsilon)}(x)),f(r_{(E,\varepsilon)}(y))\right)\leq 2\delta+(1+\delta)\|f\|_\text{Lip}d(x,y).
\end{align*}
Since $0<\delta\leq1$ is arbitrary, we conclude that $d\left((Tf)(x),(Tf)(y)\right)\leq \|f\|_\text{Lip}d(x,y)$, as desired.

Next, we show that $Tf$ is an extension of $f$ for every $f\in\text{Lip}_0(N,\mathbb{R}^n)$. Indeed, for any such $f$, for any point $x\in N$, and for any $0<\delta\leq1$, we have that the set 
$$ B_\delta = \left\{(E,\varepsilon)\in \Gamma\colon d\left(F(f,x),\widehat{f}_{(E,\varepsilon)}(x)\right)<\delta\right\}$$
is cofinal in $\Gamma$. Hence, there exists $(E,\varepsilon)\in B_\delta$ with $x\in E$ and $\varepsilon>0$ such that $(E,\varepsilon)\in B_\delta$. Now, since $x\in E\cap N$, the map $r_{(E,\varepsilon)}$ fixes the point $x$. Therefore, we obtain:

$$d\left(Tf(x),f(x)\right)=d\left(F(f,x),\widehat{f}_{(E,\varepsilon)}(x)\right)<\delta.$$
Again, making $\delta$ go to $0$ we obtain that $Tf(x)=f(x)$. 

Finally, we show that $T$ is a linear operator. Fix $f,g\in \text{Lip}_0(N)$ and $x\in M$. We want to show that $T(f+g)(x)=Tf(x)+Tg(x)$. We again follow the same idea, this time using pointwise convergence in the three coordinates $(f,x)$ and $(g,x)$ and $(f+g,x)$. Indeed, given any $0<\delta\leq1$, the set
\begin{align*}
    C_\delta = \big\{(E,\varepsilon)\in \Gamma\colon &d\left(F(f,x),\widehat{f}_{(E,\varepsilon)}(x)\right)<\delta,~d\left(F(g,x),\widehat{g}_{(E,\varepsilon)}(x)\right)<\delta\\
    & \text{and }d\left(F(f+g,x),\widehat{f+g}_{(E,\varepsilon)}(x)\right)<\delta\big\}
\end{align*}
is cofinal in $\Gamma$. Therefore, there exists $(E,\varepsilon)\in C_\delta$ such that $x\in E$ and $\varepsilon>0$. Notice as well that, by definition, $\widehat{f+g}_{(E,\varepsilon)}(x)=\widehat{f}_{(E,\varepsilon)}(x)+\widehat{g}_{(E,\varepsilon)}(x)$. This yields the estimate:
\begin{align*}
    d\big(T(f+g)(x), & Tf(x)+Tg(x)\big) \\
    &\leq d\left(F(f+g,x),\widehat{f+g}_{(E,\varepsilon)}(x)\right)+d\left(\widehat{f}_{(E,\varepsilon)}(x)+\widehat{g}_{(E,\varepsilon)}(x),F(f,x)+F(g,x)\right)\\
    &<\delta+ d\left(\widehat{f}_{(E,\varepsilon)}(x),F(f,x)\right)+d\left(\widehat{g}_{(E,\varepsilon)}(x),F(g,x)\right)< 3\delta.
\end{align*}
The linearity of $T$ now follows since $0<\delta\leq1$ can be chosen to be arbitrarily small.

\end{proof}

The previous theorem has several interesting consequences. The first one we obtain is that proper metric spaces which are local retracts are automatically Lipschitz retracts. 

\begin{corollary}\label{prop:localretra_inproper_impliLipRetr}
Let $M$ be a metric space and let $N$ be a local retract in $M$. If $N$ is proper, then $N$ is a $1$-Lipschitz retract of $M$.
\end{corollary}
\begin{proof}
    The identity map $Id\colon N\rightarrow N$ is trivially $1$-Lipschitz. Since $N$ is proper and a local retract of $M$, by Theorem \ref{Theorem:LocalRetract_ExtensionintoProper}, it can be extended to a $1$-Lipschitz map $F\colon M\rightarrow N$, which results in a $1$-Lipschitz retraction from $M$ onto $N$.
\end{proof}

The next corollary of Theorem \ref{Theorem:LocalRetract_ExtensionintoProper} contains in particular the promised converse of Remark \ref{Remark:LocCompl_implies_LocRetr_inBanach}. We obtain it by showing that local retracts in metric spaces induce local complements in their respective Lipschitz-free spaces. 

\begin{corollary}\label{prop:localretraimplilocalcomp}
Let $M$ be a metric space and let $N$ be a local retract in $M$. Then $\mathcal F(N)$ is locally complemented in $\mathcal F(M)$. 

In particular, if $X$ is a Banach space and $Y$ is a subspace of $X$, the following statements are equivalent:
\begin{enumerate}
    \item $Y$ is a local retract of $X$.
    \item $\mathcal{F}(Y)$ is locally complemented in $\mathcal{F}(X)$.
    \item $Y$ is locally complemented in $X$.
\end{enumerate}
\end{corollary}
\begin{proof}
By Theorem \ref{Theorem:LocalRetract_ExtensionintoProper}, there exists a linear extension operator $T\colon \text{Lip}_0(N)\rightarrow\text{Lip}_0(M)$ with $\|T\|=1$. Theorem \ref{Theorem:Local_Comp_Eq} (specifically condition (5)) shows that this is equivalent to $\mathcal{F}(N)$ being locally complemented in $\mathcal{F}(M)$.

For the second part of the corollary, we use Fakhoury's characterisation of local complementability. Indeed, it follows from the equivalence of (1), (5) and (6) in Theorem \ref{Theorem:Local_Comp_Eq}, that a subspace $Y$ of a Banach space $X$ is locally complemented if and only if $\mathcal{F}(Y)$ is locally complemented in $\mathcal{F}(X)$. Hence, (2) and (3) are equivalent. The first part of the corollary shows that (1) implies (2), while in Remark \ref{Remark:LocCompl_implies_LocRetr_inBanach} it is shown that (3) implies (1). This finishes the proof.
\end{proof}

We will prove that, for general metric spaces, the converse of Corollary \ref{prop:localretraimplilocalcomp} fails to be true. In order to show the strength behind the concept of local retractions let us recall a well known concept from metric space theory. We will say that a metric space $(M,d)$ is a \emph{length space} if, for every pair of points $x,y \in M$, the distance $d(x,y)$ is equal to the infimum of the length of rectifiable curves joining them. Moreover, if that infimum is always attained then we will say that $M$ is a \emph{geodesic space}. 

Length and geodesic metric spaces have been widely studied in the literature of metric spaces (c.f. e.g. \cite{bbi}). Very recently \cite{gpr18}, length metric spaces have been characterised via a geometric property of the corresponding Lipschitz-free spaces: A complete metric space $M$ is length if, and only if, $\mathcal F(M)$ has the \textit{Daugavet property}, which means that every rank one continuous operator $T:\mathcal F(M)\rightarrow \mathcal F(M)$ satisfies that $\Vert T+\mathrm{Id}\Vert=1+\Vert T\Vert$ (see \cite{gpr18} and references therein for background on the Daugavet property).

The following result shows that the property of being length is inherited by local retracts.

\begin{proposition}\label{prop:lengthlocalretract}
    Let $M$ be a complete metric space and $N$ be a closed subspace which is a local retract of $M$. If $M$ is length, so is $N$.
\end{proposition}

\begin{proof}
Since $N$ is a complete metric space, it suffices to prove that, given any pair of points $x,y\in N$ with $x\neq y$ and any $\varepsilon>0$, it follows $B\left(x,\frac{d(x,y)}{2}+\varepsilon\right)\cap B\left(y,\frac{d(x,y)}{2}+\varepsilon\right)\neq \emptyset$ (see e.g. \cite[Theorem 2.4.16]{bbi}).

So take $x,y\in N$ with $x\neq y$ and $\varepsilon>0$, and select $\delta>0$ small enough to guarantee  $(1+\delta)\left(\frac{d(x,y)}{2}+\delta\right)\leq \frac{d(x,y)}{2}+\varepsilon$. Since $M$ is length we have that there exists $z\in M$ such that $d(x,z)\leq \frac{d(x,y)}{2}+\delta$ and $d(y,z)\leq \frac{d(x,y)}{2}+\delta$.

Now set $E:=\{x,y,z\}\subseteq M$. Since $N$ is a local retract in $M$ there exists a Lipschitz map $r:E\rightarrow N$ such that $r(x)=x, r(y)=y$ and $\Vert r\Vert\leq 1+\delta$. We claim that $r(z)\in B(x,\frac{d(x,y)}{2}+\varepsilon)\cap B(y,\frac{d(x,y)}{2}+\varepsilon)$. Indeed, observe that
$$d(x,r(z))=d(r(x),r(z))\leq (1+\delta)d(x,z)\leq (1+\delta)\left(\frac{d(x,y)}{2}+\delta\right)\leq \frac{d(x,y)}{2}+\varepsilon$$
by the choice of $\delta$. The proof that $d(r(z),y)\leq \frac{d(x,y)}{2}+\varepsilon$ is similar and finishes the proof.
\end{proof}

\begin{remark}
The above result should be compared with the fact that the Daugavet property is not inherited by taking $1$-complemented subspaces. For instance, $L_1([0,1])$ has the Daugavet property but its complemented subspace $\ell_1$ fails it (c.f. e.g. \cite{werner01}).
\end{remark}

We can use Proposition \ref{prop:lengthlocalretract} to show that the converse of Corollary \ref{prop:localretraimplilocalcomp} does not hold for general metric spaces. 

\begin{example}
Let $M:=[0,1]$ and $N:=\{0,1\}$. It is immediate that $\mathcal F(N)=\mathbb R$ is even $1$-complemented in $M$. However, Proposition \ref{prop:lengthlocalretract} shows that $N$ is not a local retract in $M$.
\end{example}

We turn our attention now to the second notion we have defined in this section: ai-local retracts. Since this property is stronger than the notion of local retract, the statements of both Theorem \ref{Theorem:LocalRetract_ExtensionintoProper} and Proposition \ref{prop:lengthlocalretract} also hold for ai-local retracts. However, we can obtain the following strengthening of Theorem \ref{Theorem:LocalRetract_ExtensionintoProper} with essentially the same proof, which we include for the convenience of the reader.
\begin{theorem}
\label{Theorem:AIretract_extensionIsometries_Proper}
    Let $M$ be a metric space and $N$ be an ai-local retract in $M$. Then, if $X$ is a proper metric space, for every isometry $f:N\rightarrow X$ there exists an extension $F:M\rightarrow X$ of $f$ such that $F$ is an isometry as well.
\end{theorem}
\begin{proof}

We proceed as in the proof of \ref{Theorem:LocalRetract_ExtensionintoProper}. Call $\Gamma:=\{(E,\varepsilon): E\subseteq M\mbox{ is finite and } 0<\varepsilon\leq1\}$, and endow $\Gamma$ with the same partial order $\leq$: $(E,\varepsilon)\leq (F,\delta)$ if and only if $E\subseteq F$ and $\delta\leq \varepsilon$. With this order, $\Gamma$ is a directed set. 

Given $(E,\varepsilon)\in \Gamma$, since $N$ is an ai-local retract in $M$, there exists a map $r_{(E,\varepsilon)}:E\rightarrow N$ satisfying that $r_{(E,\varepsilon)}(e)=e$ for every $e\in E\cap N$ and such that $(1-\varepsilon)d(x,y)\leq d(r_{(E,\varepsilon)}(x),r_{(E,\varepsilon)}(y))\leq (1+\varepsilon)d(x,y)$ for all $x,y\in E$. For every isometry $f\colon N\rightarrow X$, define a map $\widehat{f}_{(E,\varepsilon)}\colon M\rightarrow X$ by:
$$\widehat{f}_{(E,\varepsilon)}(x):=\left\{\begin{array}{cc}
 f(r_{(E,\varepsilon)}(x))    & \mbox{ if }x\in E,  \\
    0 & \mbox{otherwise}. 
\end{array} \right.$$

For a fixed $(E,\varepsilon)\in\Gamma$, the point $\widehat{f}_{(E,\varepsilon)}(x)$ belongs to the compact set $B(0,(1+\varepsilon))$ for every $x\in M$. Therefore, the set 
$$ \{(\widehat{f}_{(E,\varepsilon)}(x))_{x\in M}\colon (E,\varepsilon)\in \Gamma\},$$
indexed by the directed set $\Gamma$, is a net in the product space $\prod_{x\in M}B(0,2d(x,0))$, which is compact by Tychonoff's Theorem. 

Using compactness, consider a cluster point 
$$F= (F_{(x)})_{x\in M}\in \prod_{x\in M}B(0,2d(x,0))$$
of the previously defined net, which is a map $F\colon M\rightarrow X$. Arguing as in the proof of Theorem \ref{Theorem:LocalRetract_ExtensionintoProper}, we obtain that $F$ is an extension of $f$. 

We finish the proof by showing that $F$ is an isometry. Fix two points $x,y\in M$. Given any $0<\delta\leq 1$, by definition of cluster point, and since the product topology is the topology of pointwise convergence, we have that the set 
$$ A_\delta=\left\{(E,\varepsilon)\in \Gamma\colon d\left(F(x),\widehat{f}_{(E,\varepsilon)}(x)\right)<\delta\text{ and }d\left(F(y),\widehat{f}_{(E,\varepsilon)}(y)\right)<\delta\right\}$$
is cofinal in $\Gamma$. Hence, given $\left(\{x,y\},\delta\right)\in \Gamma$, there exists $(E,\varepsilon)\in A_\delta$ such that $\{x,y\}\subset E$ and $\varepsilon\leq \delta$. By definition of $\widehat{f}_{(E,\varepsilon)}(x)$, and since $f$ is an isometry, we obtain that
\begin{align*}
    d\left(F(x),F(y)\right)&\leq d\left(\widehat f_{(E,\varepsilon)}(x),\widehat{f}_{(E,\varepsilon)}(y)\right)+d\left(F(x),\widehat f_{(E,\varepsilon)}(x)\right)+d\left(F(y),\widehat f_{(E,\varepsilon)}(y)\right)\\
    &<d\left(\widehat r_{(E,\varepsilon)}(x),r_{(E,\varepsilon)}(y)\right)+2\delta\leq (1+\varepsilon)d(x,y)+2\delta.
\end{align*}
Similarly, we have:
\begin{align*}
    d\left(F(x),F(y)\right)&\geq d\left(\widehat f_{(E,\varepsilon)}(x),\widehat{f}_{(E,\varepsilon)}(y)\right)-d\left(F(x),\widehat f_{(E,\varepsilon)}(x)\right)-d\left(F(y),\widehat f_{(E,\varepsilon)}(y)\right)\\
    &\geq d\left(\widehat r_{(E,\varepsilon)}(x),r_{(E,\varepsilon)}(y)\right)-2\delta\geq (1-\varepsilon)d(x,y)-2\delta.
\end{align*}
Since $0<\delta\leq 1$ is arbitrary and $\varepsilon<\delta$, we obtain that $F$ is an isometry.
\end{proof}

\section{Absolute (ai) local retracts}

In this section we study which metric spaces are local retracts, or even ai-local retracts, in every metric space which contains them. This kind of universality concepts are well studied for linear projections, local complements and Lipschitz retracts. 

Recall that a metric space $M$ is called an \textit{absolute $1$-Lipschitz retract} if it is a $1$-Lipschitz retract in every metric space containing it. Analogously, we may define the following concepts, which are the main object of study of this section:

\begin{definition}\label{defi:absolutelocal}
    Let $M$ be a metric space.
\begin{enumerate}
    \item We say that $M$ is \textit{an absolute local retract} if it is a local retract in every metric space containing it.
    \item We say that $M$ is \textit{an absolute ai-local retract} if it is an ai-local retract in every metric space containing it.
\end{enumerate}
\end{definition}

Before studying conditions which characterise absolute local retracts and absolute ai-local retracts, it will be useful to recall some classical characterisations for absolute $1$-Lipschitz retracts. We say that a metric space $M$ is \emph{metrically convex} if for every pair of points $x_1,x_2\in M$, the closed balls $B(x_1,\delta_1)$ and $B(x_2,\delta_2)$ intersect whenever $d(x_1,x_2)\leq \delta_1+\delta_2$. We say that a metric space $M$ has the \emph{binary intersection property} if any arbitrary collection of mutually intersecting closed balls $\{B(x_i,\delta_i)\}_{i\in I}$ has nonempty intersection. 

We refer to Chapter 1 in \cite{BenLin00} (specifically Propositions 1.2 and 1.4) for the following characterisation:

\begin{theorem}
\label{theo:caraAbsLipRetr}
    Let $M$ be a metric space. The following are equivalent:
    \begin{enumerate}
        \item $M$ is an absolute $1$-Lipschitz retract.
        \item $M$ is geodesic and has the binary intersection property.
        \item For every pair of metric spaces $Y\subseteq X$ and every Lipschitz mapping $f:Y\rightarrow M$, there exists an extension $F:X\rightarrow M$ with $F_{|N}=f$ and $\Vert F\Vert=\Vert f\Vert$.
    \end{enumerate}
\end{theorem}

Note that absolute $1$-Lipschitz retracts are thus characterised both by a purely geometric condition (condition (2)), and also in terms of extension of Lipschitz maps from arbitrary metric spaces (condition (3)). The geometric condition (2) has been known in the literature as \emph{hyperconvexity}, following \cite{arpa1956}, where the equivalence between (2) and (3) can already be found. Metric spaces enjoying the extension property (3) are usually called \emph{1-injective}, since this nonlinear concept extends the classical one from Banach spaces theory. 

Let us also briefly sketch the usual proof of the equivalence between (1) and (3) in the previous theorem: We start by showing that $\ell_\infty(\Gamma)$ is $1$-injective, by simply applying McShane's Extension Theorem to each coordinate. Then, we isometrically embed the metric space $M$ into $\ell_\infty(\Gamma)$ for large enough $\Gamma$. If $M$ is an absolute $1$-Lipschitz retract, then there exists a $1$-Lipschitz retraction $R$ from $\ell_\infty(\Gamma)$ onto $M$. Now, given a pair of metric spaces $Y\subset X$ and any Lipschitz function $f\colon Y\rightarrow M$, we can extend $f$ to a function $\widehat{f}\colon X\rightarrow\ell_\infty(\Gamma)$ with the same Lipschitz constant. The composition $F= R\circ \widehat{f}\colon X\rightarrow M$ gives an extension of $f$ into $M$ with the same Lipschitz constant too, showing that $M$ is $1$-injective. Conversely, if $M$ is $1$-injective, then the identity map $id\colon M\rightarrow M$ can be extended to a $1$-Lipschitz retraction in any metric space where $M$ is contained, which implies that $M$ is an absolute $1$-Lipschitz retract.

\subsection{Absolute local retracts}

Absolute local retracts can also be characterised both geometrically and in terms of extensions of Lipschitz maps. For the geometric condition, the binary intersection property needs to be weakened to a finite version: We say that a metric space $M$ has the \emph{finite binary intersection property} if any finite collection of mutually intersecting closed balls $\{B(x_i,\delta_i)\}_{i=1}^n$ has nonempty intersection. Similarly, $1$-injectivity in Theorem \ref{theo:caraAbsLipRetr} will be replaced by a finite version: A metric space $M$ is called \emph{finitely $1$-injective} if every for every pair of finite metric spaces $Y\subseteq X$ and every Lipschitz map $f\colon Y\rightarrow M$ there exists an extension $F:X\rightarrow M$ with $F_{|N}=f$ and $\Vert F\Vert=\Vert f\Vert$.  

As shown in \cite{arpa1956}, metric spaces which are geodesic and enjoy the finite binary intersection property are precisely those which are finitely $1$-injective. This, together with some usual techniques, allows us to obtain the characterisation of absolute local retracts we sought:

\begin{theorem}\label{theo:carahyperconvexelocalretracevery}
Let $M$ be a complete metric space. The following assertions are equivalent:
\begin{enumerate}
    \item $M$ is an absolute local retract.
    \item $M$ is length and for any finite collection of mutually intersecting balls $\{B(x_i,\delta_i)\}_{i=1}^n$ and $\varepsilon>0$, the intersection of $\{B(x_i,\delta_i+\varepsilon)\}_{i=1}^n$ is nonempty.
    \item $M$ is geodesic and has the finite binary intersection property.
    \item $M$ is finitely $1$-injective, i.e.: For every pair of finite metric spaces $Y\subseteq X$, every Lipschitz mapping $f:Y\rightarrow M$ there exists an extension $F:X\rightarrow M$ with $F_{|N}=f$ and $\Vert F\Vert=\Vert f\Vert$.
\end{enumerate}
\end{theorem}

\begin{proof}

    (1)$\Rightarrow$(2) Suppose that $M$ is a local retract in every metric space containing it. Consider a large enough cardinal $\Gamma$ such that $M$ is isometric to a subset of $\ell_\infty(\Gamma)$, and write $M\subset\ell_\infty(\Gamma)$. Then $M$ is a local retract of $\ell_\infty(\Gamma)$. Since every Banach space is a length space, Proposition \ref{prop:lengthlocalretract} implies that $M$ is length. 

    Now, let $\{B(x_i,\delta_i)\}_{i=1}^n$ be a finite collection of mutually intersecting closed balls in $M$. We can consider these balls in $\ell_\infty(\Gamma)$, which has the binary intersection property. In particular, there exists a point $p\in \ell_\infty$ with $d(x_i,p)\leq \delta_i$ for all $i=1,\dots,n$. Write $Y=\{x_i\}_{i=1}^n$, and $X=Y\cup \{p\}$, which are finite subsets of $\ell_\infty(\Gamma)$. Given $\varepsilon>0$ and $\delta_0=\max\{\delta_i\colon i=1,\dots,n\}$, there exists a $\left(1+\frac{\varepsilon}{\delta_0}\right)$-Lipschitz map $r\colon X\rightarrow M$ such that $r(x_i)=x_i$ for all $i=1,\dots,n$. Since $r$ is $\left(1+\frac{\varepsilon}{\delta_0}\right)$-Lipschitz, we have that 
    $$d(r(p),x_i)\leq \left(1+\frac{\varepsilon}{\delta_0}\right)d(p,x_i)\leq\left(1+\frac{\varepsilon}{\delta_0}\right)\delta_i\leq \delta_i+\varepsilon.$$
    We conclude that $r(p)$ is a point in $M$ belonging to the intersection of the collection $\{B(x_i,\delta_i+\varepsilon)\}_{i=1}^n$.

    (2)$\Rightarrow$(3) We start showing that $M$ has the finite intersection property. Let $\{B(x_i,\delta_i)\}_{i=1}^n$ be a finite collection of mutually intersecting closed balls in $M$. Choose a decreasing sequence $\{\varepsilon_k\}_{k\in\mathbb{N}}$ of strictly positive real numbers converging to $0$ such that $\frac{3}{4}\sum_{l\in\mathbb{N}}\varepsilon_{k+l}<\varepsilon_{k}$.    
    
    By induction, we will construct a sequence of points $(y_k)_{k\in\mathbb{N}}\subset M$ such that $y_k$ belongs to \\ $\bigcap_{i=1}^n B\left(x_i,\delta_i+\varepsilon_k\right)$ for all $k\in\mathbb{N}$, and such that $d(y_k,y_{k-1})\leq \frac{3}{2}\varepsilon_{k-1}$ for $k\geq 2$. 

    By assumption, there exists $y_1\in \bigcap_{i=1}^n B\left(x_i,\delta_i+\varepsilon_1\right)$, so the first step in the induction is clear. Suppose we have constructed $y_{k-1}$ for some $k\geq 2$ with the desired properties. Since $M$ is length and for every $i\in\{1,\dots,n\}$ we have that  $d(x_i,y_{k-1})< \delta_i+\varepsilon_{k-1}+\frac{\varepsilon_{k}}{2}$, the closed balls $B\left(x_i,\delta_i+\frac{\varepsilon_{k}}{2}\right)$ and $B\left(y_{k-1},\varepsilon_{k-1}\right)$ have nonempty intersection. Therefore, there exists a point $y_k\in M$ which belongs to $B\left(x_i,\delta_i+\varepsilon_{k}\right)$ for every $i=1,\dots,n$, and also belongs to $B\left(y_{k-1},\varepsilon_{k-1}+\frac{\varepsilon_k}{2}\right)$. Since $\varepsilon_k\leq \varepsilon_{k-1}$, this finishes the induction. 

    Now, given $k,j\in\mathbb{N}$ with $k\geq 2$, we have by the triangle inequality and by the choice of the sequence $\{\varepsilon_k\}_{k\in\mathbb{N}}$ that
    $$ d(y_k,y_{k+j})\leq \sum_{l=1}^j d(y_{(k-1)+l},y_{(k-1)+l+1})\leq \frac{3}{2}\sum_{l=1}^j \varepsilon_{(k-1)+l}<\varepsilon_{k-1}.$$
    Therefore, the sequence $\{y_k\}_{k\in\mathbb{N}}$ is Cauchy, and converges to a point $y$ in the complete metric space $M$. Using again the previous estimate, we obtain that for $i=1,\dots,n$, it holds that $d(x_i,y_{k+j})<\delta_i+\varepsilon_k+\varepsilon_{k-1}$ for all $k\geq 2$ and all $j\in\mathbb{N}$. Hence, since the sequence $\{\varepsilon_k\}_{k\in\mathbb{N}}$ converges to $0$, we have that $d(x_i,y)\leq \delta_i$ for all $i=1,\dots,n$. We conclude that $M$ has the finite intersection property. 

    It remains to show that $M$ is geodesic. As we have shown that $M$ has the finite intersection property, by Theorem 4.5 and the subsequent remark in \cite{linds64}, we have that if $\{B(x_i,\delta_i)\}_{i\in I}$ is an arbitrary collection of mutually intersecting closed balls such that the set $\{x_i\}_{i\in I}$ is relatively compact in $M$, then $\bigcap_{i\in I}B(x_i,\delta_i)$ is nonempty. Now, using that $M$ is length and the previous result, we have that given two different points $x$ and $y$ in $M$, there exists a point $z$ in the set
    $$\bigcap_{\varepsilon>0}B\left(x,(1+\varepsilon)\frac{d(x,y)}{2}\right)\cap \bigcap_{\varepsilon>0}B\left(y,(1+\varepsilon)\frac{d(x,y)}{2}\right).$$
    This implies that $d(x,z)=d(y,z)=\frac{1}{2}d(x,y)$, which shows that $M$ is geodesic. 

    (3)$\Leftrightarrow$(4) This follows from Theorem 2 in Section 2 and Theorem 2 in Section 3 of \cite{arpa1956}.

    (4)$\Rightarrow$(1) It is immediate that if (4) is satisfied, then $M$ is a local retract of any metric space $X$ containing it, since given any finite set $E\subset X$, the identity map in $E\cap M$ can be extended to a $1$-Lipschitz map $r\colon E\rightarrow M$ which fixes every point in $E\cap M$. 
    
\end{proof}

Observe that a Banach space $X$ has property (4) in Theorem \ref{theo:carahyperconvexelocalretracevery} if, and only if $X^*$ is isometrically an $L_1(\mu)$ space (c.f. e.g. \cite[Theorem 3.5]{rueda21}). In particular, the space $c_0$ is an example of a (separable) Banach space which is an absolute local retract. Moreover, for Banach spaces the finite intersection property can be simplified: it follows from \cite{linds64} (Theorem 6.1 (12) and Theorem 4.3) that $X^*$ is an $L_1(\mu)$ space (or, equivalently, satisfies (1)-(4) in the previous theorem) if and only if every collection of 4 closed balls $\{B(x_i,\delta_i)\}_{i=1}^4$ in $X$ which mutually intersect, has nonempty intersection.

\subsection{Absolute ai-local retracts}

Let us end the section by characterising ai-local retracts in terms of extensions of certain Lipschitz maps. In order to explain the idea behind this, observe that from the results of \cite[Section 4]{aln2} it is proved that a Banach space $X$ is a Gurarii space if, and only if, $X$ is almost isometric ideal in every Banach space containing it.

Let us expand on the strategy to characterise ai-local retracts using Gurarii spaces. As we sketched after the statement of the classical Theorem \ref{theo:caraAbsLipRetr}, a usual proof of showing that an absolute $1$-Lipschitz retract $M$ is $1$-injective consists of isometrically embedding $M$ into an $\ell_\infty(\Gamma)$ space, which is known to be $1$-injective. Then, a Lipschitz retraction from $\ell_\infty(\Gamma)$ onto $M$ allows to reduce the range of extensions to the space $M$. Hence, in order to characterise an absolute ai-local retract $M$ in terms of extensions of isometries, a natural strategy is to embed $M$ into another metric space which is known to have the isometry extension property we seek to prove in $M$. A Banach space $X$ is a \emph{Gurarii space} if given $\varepsilon>0$, given any pair of finite dimensional Banach spaces $Y$ and $Z$, and given a pair of isometric embeddings $S\colon Y\rightarrow Z$ and $T\colon Y\rightarrow X$, there exists a $(1+\varepsilon)$-isometry $\widehat T\colon Z\rightarrow X$ such that $\widehat T\circ S=T$. By \cite[Theorem 3.6]{gk}, every Banach space (and, in particular, every metric space) can be isometrically embedded into a Gurarii space, and thus Gurarii spaces are good candidates for our purposes. However, we first need to use properties of Lipschitz-free spaces to obtain an isometry extension statement for isometries between finite metric spaces instead of linear isometries between finite dimensional Banach spaces:

\begin{lemma}
\label{metricGurarii}
    Let $M$ be a metric space. Then, $M$ isometrically embeds into a Gurarii space $X$ with an isometric embedding $\iota_M\colon M\rightarrow X$ such that for every $\varepsilon>0$, for every pair of finite metric spaces $E$ and $F$, and for every pair of isometries $S:E\rightarrow F$ and $T:E\rightarrow M$ there exists a $(1+\varepsilon)$-isometry $\widehat T: F\rightarrow X$ such that $\widehat T\circ S= \iota_M\circ T$.
\end{lemma}
\begin{proof}

    Let $\delta_M\colon M\rightarrow\mathcal{F}(M)$, $\delta_E\colon E\rightarrow \mathcal{F}(E)$ and $\delta_F\colon F\rightarrow \mathcal{F}(F)$ be the isometric embeddings of $M, E$ and $F$ into their respective Lipschitz-free spaces, given by the Dirac map.
    
    By \cite[Theorem 3.6]{gk}, there exists a Gurarii space $X$ such that $\mathcal{F}(M)$ embeds linearly and isometrically into $X$. For simplicity, we may consider $\mathcal{F}(M)$ as a linear subspace of $X$, and thus $\iota_M=\delta_M$ is an isometric embedding of $M$ into $X$. Now, given a pair of finite metric spaces $E$ and $F$, and isometries $S:E\rightarrow F$ and $T:E\rightarrow M$, the linearisation property of Lipschitz-free spaces yields two linear isometries $\widetilde S\colon \mathcal{F}(E)\rightarrow\mathcal{F}(F)$ and $\widetilde T\colon \mathcal{F}(E)\rightarrow \mathcal{F}(M)$, such that $\widetilde S\circ \delta_E = \delta_F\circ S$ and $\widetilde T\circ \delta_E=\delta_M\circ T$. The map $\widetilde T$ is in particular an isometry into the Gurarii space $X$.

    Since the Banach spaces $\mathcal{F}(E)$ and $\mathcal{F}(F)$ are finite dimensional, given $\varepsilon>0$, there exists a linear $(1+\varepsilon)$-isometry $T^*\colon \mathcal{F}(F)\rightarrow X$ such that $T^*\circ \widetilde S = \widetilde T$. The map $\widehat T=T^*\circ \delta_F\colon F\rightarrow X$ is the $(1+\varepsilon)$-isometry we sought. Indeed:

    $$\widehat T\circ S=T^*\circ\widetilde S\circ \delta_E=\widetilde T\circ \delta_E=\delta_M\circ T =\iota_M\circ T.$$
\end{proof}

With this result, we can now prove the following characterisation:

\begin{theorem}\label{theo:carafiniinyecairetracevery}
Let $M$ be a metric space. The following are equivalent:
\begin{enumerate}
    \item $M$ is an absolute ai-local retract.
    \item For every $\varepsilon>0$, for every pair of finite metric spaces $E$ and $F$, and for every pair of isometries $S:E\rightarrow F$ and $T:E\rightarrow M$ there exists a mapping $\widehat T: F\rightarrow M$ such that $\widehat T\circ S=T$ and such that
    $$(1-\varepsilon)d(x,y)\leq d(\widehat T(x),\widehat T(y))\leq (1+\varepsilon)d(x,y).$$
\end{enumerate}
\end{theorem}

\begin{proof}
(1)$\Rightarrow$(2). Let $\varepsilon>0$, let $E, F$ be two finite metric spaces and let $S:E\rightarrow F$ and $T:E\rightarrow M$ be a pair of isometries. Fix an arbitrary $\rho>0$. By Lemma \ref{metricGurarii}, we may regard $M$ as a subset of a Gurarii space $X$, and we may consider a $(1+\rho)$-isometry $T^*\colon F\rightarrow X$ such that $T^* \circ S = T$. Now, since $M$ is an ai-retract of $X$, there exists a Lipschitz map $r\colon T^*(F)\rightarrow M$ such that $r(x)=x$ for all $T^*(F)\cap M$, and such that $(1-\rho)d(x,y)\leq d(r(x),r(y))\leq (1+\rho)d(x,y)$. Note that since $T=T^*\circ S$, the image of $T$ is contained in $T^*(F)\cap M$, and we have that $r\circ T= T$.

Defining $\widehat T= r\circ T^*\colon F\rightarrow M$, we obtain that $\widehat T\circ S= r \circ T^* \circ S = r\circ T =T$ and
$$ (1-\rho)^2d(x,y)\leq d(\widehat T(x),\widehat T(y))\leq (1+\rho)^2 d(x,y).$$
Since $\rho>0$ is arbitrary, the conclusion of (2) follows.

(2)$\Rightarrow$(1). This proof is rather simpler. Let $X$ be any metric space containing $M$ and let us prove that $M$ is an ai-local retract in $X$. 

In order to do so, pick a finite subset $E$ of $X$ and $\varepsilon>0$, and let us construct a $(1+\varepsilon)$-isometry $T:E\rightarrow M$ fixing $E\cap M$. Up to adding an element of $M$ to $E$ we can assume with no loss of generality that $E\cap M\neq \emptyset$. Now set the inclusion operator $i:E\cap M\rightarrow M$ and the inclusion operator $j:E\cap M\rightarrow E$. By (2) there exists an almost isometric Lipschitz mapping $T:E\rightarrow M$ such that $T\circ j=i$. It is immediate that $T$ is the desired almost isometric retraction.
\end{proof}

\begin{remark}\label{remark:polishurysohn}
Observe that condition (2) in Theorem \ref{theo:carafiniinyecairetracevery} appeared in \cite{Mell06} in connection with the \textit{extension property} described there. In \cite[Section 3]{Mell06} it is shown that the only Polish metric space with the above property is the \textit{Urysohn space} $\mathbb U$ (see \cite{Mell06} and references therein). As a consequence, there is no separable Banach space which is an ai-retract in every metric space containing it. In particular, $c_0$ is an example of a Banach space which is an absolute local retract, but not an absolute ai-local retract.
\end{remark}

\section{Existence of ai-local retracts in metric spaces}

In this last section, we show that every metric space has a rich structure of ai-local retracts for any prescribed density character. As discussed in the introduction, this strengthens Theorem 5.3 of \cite{HajQui22}, since every ai-local retract is in particular a local retraction, and thus, by Theorem \ref{Theorem:LocalRetract_ExtensionintoProper}, locally complemented as a metric space in the sense of \cite{HajQui22}.

Let us start with a couple of elementary lemmata. The first lemma allows us to work on dense subsets of metric spaces. 

\begin{lemma}\label{lemma:airetractdensigran}
Let $M$ be a complete metric space and let $N$ be a closed subset of $M$. Consider two dense subsets $D$ of $M$ and $S$ of $N$ with $S\subseteq D$. If $S$ is an (ai) local retract in $D$ then $N$ is an (ai) local retract in $M$.
\end{lemma}

\begin{proof}
We will prove the case of ai-local retract because the proof will cover the other case. 

The proof will be divided in two steps:

\textbf{Step 1:} $S$ is an ai-local retract in $M$.

\begin{proof}[Proof of Step 1.]

Let $E\subseteq M$ be finite and let $\varepsilon>0$, and let us find a mapping $T:E\rightarrow S$ such that $T(e)=e$ holds for $e\in E\cap S$ and $(1-\varepsilon)d(x,y)\leq d(T(x),T(y))\leq (1+\varepsilon)d(x,y)$ holds for every $x,y\in E$.

In order to do so, consider $\theta:=\min_{x\neq y\in E} d(x,y)>0$ and select $\delta>0$ and $\eta>0$ small enough to get $(1-\varepsilon)<(1-\delta)\left(1-\frac{2\eta}{\theta}\right)$ and $(1+\delta)\left(1+\frac{2\eta}{\theta}\right)<1+\varepsilon$. Call $\widetilde E:=E\cap D$ and set $E\setminus D:=\{e_1,\ldots, e_p\}$ for some $p\in\mathbb N$. Since $D$ is dense in $M$ we can find, for $1\leq i\leq p$, an element $e_i'\in D$ such that $d(e_i,e_i')<\eta$. Since $S$ is an ai-local retract in $D$ there exists a map $\widetilde T:\widetilde E\cup\{e_1',\ldots, e_p'\}\rightarrow S$ such that $\widetilde T(e)=e$ for every $e\in \widetilde E\cup\{e_1',\ldots, e_n'\}\cap S$ and $(1-\delta)d(x,y)\leq d(\widetilde T(x),\widetilde T(y))\leq (1+\delta)d(x,y)$ holds for every $x,y\in \widetilde E\cup\{e_1',\ldots, e_p'\}$. Now our desired mapping $T:E\rightarrow S$ is defined by the equation
$$T(x):=\left\{\begin{array}{cc}
    \widetilde T(e_i') &\mbox{ if } x=e_i\mbox{ for some i}, \\
     \widetilde T(x) &\mbox{otherwise.} 
\end{array} \right.$$
First, observe that $E\cap S\subseteq E\cap D$. Hence, given any $e\in E\cap S$ we get
$$T(e)=\widetilde T(e)=e.$$
In order to finish the proof take $x,y\in E$, and let us prove that $(1-\varepsilon)d(x,y)\leq d(T(x),T(y))\leq (1+\varepsilon)d(x,y)$. In order to do so, let us distinguish between three cases:
\begin{enumerate}
    \item If $x,y\in E\cap D$ we have $T(x)=\widetilde T(x)$ and $T(y)=\widetilde T(y)$, from where the inequalities to be proved are satisfied by the properties of the mapping $\widetilde T$.
    \item Assume $x=e_i$ for $1\leq i\leq n$ and $y\in E\cap D$ (the other case is similar). For the inequality from above write
    $$d(T(x),T(y))=d(\widetilde T(e_i'),\widetilde T(y))\leq (1+\delta)d(e_i',y)\leq (1+\delta)(d(e_i,y)+d(e_i',e_i))\leq (1+\delta)(d(x,y)+\eta).$$
    Taking into account that $\theta\leq d(x,y)$ we infer
    $$(1+\delta)(d(x,y)+\eta)\leq (1+\delta)\left(d(x,y)+\frac{\eta}{\theta}d(x,y)\right)\leq (1+\delta)\left(1+\frac{\eta}{\theta}\right) d(x,y).$$
    Similarly, for an inequality from below we get
    $$d(T(x),T(y))=d(\widetilde T(e_i'),\widetilde T(y))\geq (1-\delta)d(e_i',y)\geq (1-\delta)(d(x,y)-d(e_i,e_i'))\geq (1-\delta)(d(x,y)-\eta).$$
    Once again taking into account that $\theta\leq d(x,y)$ we get
    $$(1-\delta)(d(x,y)-\eta)\geq (1-\delta)\left(d(x,y)-\frac{\eta}{\theta} d(x,y)\right)\geq (1-\delta)\left(1-\frac{\eta}{\theta}\right)d(x,y).$$
    \item Finally, let us assume that $x=e_i$ and $y=e_j$ for $i\neq j$. Similar estimates to those of the case (2) prove that 
    $$(1-\delta)\left(1-\frac{2\eta}{\theta}\right)d(x,y)\leq d(T(x),T(y))\leq (1+\delta)\left(1+\frac{2\eta}{\theta}\right)d(x,y)$$
    \end{enumerate}
By the above discussion of cases we derive that 
$$(1-\delta)\left(1-\frac{2\eta}{\theta}\right)d(x,y)\leq d(T(x),T(y))\leq (1+\delta)\left(1+\frac{2\eta}{\theta}\right)d(x,y)$$
holds for every $x,y\in E$. The proof of Step 1 is finished by conditions behind the choice of $\delta$ and $\eta$.
\end{proof}

\textbf{Step 2:} $N$ is an ai-local retract in $M$.

In order to prove this take $E\subseteq M$ finite and $\varepsilon>0$ and let us find $T:E\rightarrow N$ satisfying that $T(e)=e$ for every $e\in E\cap N$ and $(1-\varepsilon)d(x,y)\leq d(T(x),T(y))\leq (1+\varepsilon)d(x,y)$ holds for every $x,y\in E$.

Once again call $\theta:=\min\limits_{x\neq y\in E} d(x,y)>0$ and choose $\delta>0$ and $\eta>0$ such that 
$$(1-\varepsilon)<\left((1-\delta)\left(1-\frac{2\eta}{\theta}\right)-\frac{2\eta}{\theta}\right)\leq \left((1+\delta)\left(1+\frac{2\eta}{\theta}\right)+\frac{2\eta}{\theta}\right)<1+\varepsilon.$$

Define $\widetilde E:=E\cap (N\setminus S)=\{e_1,\ldots, e_p\}$ for certain $p\in\mathbb N$. For every $1\leq i\leq p$ there exists by a density argument an element $e_i'\in S$ such that $d(e_i,e_i')<\eta$.

Consider $F:=E\cup\{e_1',\ldots, e_p'\}\subseteq M$. Since $F$ is finite there exists, since $S$ is an ai-local retract in $M$, an operator $\widetilde T:F\rightarrow S$ such that $\widetilde T(x)=x$ for every $x\in F\cap S$ and $(1-\delta)d(x,y)\leq d(\widetilde T(x),\widetilde T(y))\leq (1+\delta)d(x,y)$ holds for every $x,y\in F$. 

Finally, define $T:E\rightarrow S$ by the equation
$$T(x):=\left\{ \begin{array}{cc}
    e_i & \mbox{if }x=e_i\mbox{ for some }1\leq i\leq p, \\
    \widetilde T(x) &\mbox{ otherwise}. 
\end{array} \right .$$
First, observe that $T(x)=x$ if $x\in E\cap N$. Indeed, the very definition of $T$ shows that the above equality is clear if $x=e_i$ for some $1\leq i\leq p$. Otherwise, $x\in E\cap S$, from where $T(x)=\widetilde T(x)=x$.

In order to finish the proof take $x,y\in E$ and let us prove that $(1-\varepsilon)d(x,y)\leq d(T(x),T(y))\leq (1+\varepsilon)d(x,y)$. Let us divide the proof by cases:
\begin{enumerate}
    \item If $x,y\notin \{e_1,\ldots, e_p\}$ we get $T(x)=\widetilde T(x)$ and $T(y)=\widetilde T(y)$. Hence
    $$(1-\delta)d(x,y)\leq d(T(x),T(y))\leq (1+\delta)d(x,y)$$
    follows by the property defining $\widetilde T$.
    \item If $x=e_i$ for some $1\leq i\leq p$ and $y\notin \{e_1,\ldots, e_p\}$ we have

    \begin{align*}
    d(T(x),T(y))\leq d(e_i,e_i')+d(e_i',\widetilde T(y))& <\eta+d(\widetilde T(e_i'),\widetilde T(y))\\ &  \leq \eta+(1+\delta)d(e_i',y)\\
    & \leq \eta+(1+\delta)(d(e_i,y)+d(e_i,e_i'))\\
    & \leq \eta+(1+\delta)(d(x,y)+\eta)\\
& \leq \frac{\eta}{\theta}d(x,y)+(1+\delta)(d(x,y)+\frac{\eta}{\theta}d(x,y))\\
& \leq \left(\frac{\eta}{\theta}+(1+\delta)\left( 1+\frac{\eta}{\theta}\right)\right) d(x,y).  
    \end{align*}
Similar estimates involving inequalities from below allow to prove
$$d(T(x),T(y))\geq \left((1-\delta)\left(1-\frac{\eta}{\theta}\right)-\frac{\eta}{\theta}\right) d(x,y).$$
\item If $x=e_i$ and $y=e_j$ for certain $i,j\in\{1,\ldots, p\}$, using again similar arguments as in the previous case, we arrive at
$$\left((1-\delta)\left(1-\frac{2\eta}{\theta}\right)-\frac{2\eta}{\theta}\right) d(x,y)\leq d(T(x),T(y))\leq \left((1-\delta)\left(1+\frac{2\eta}{\theta}\right)+\frac{2\eta}{\theta}\right) d(x,y)$$
\end{enumerate}
The choice of $\delta$ and $\theta$ now yield the desired inequalities.
\end{proof}

The next lemma is a generalisation of \cite[Lemma 5.1]{HajQui22} in the sense of obtaining inequalities from above. This is in turn a metric version of \cite[Lemma 2.1]{Abrahamsen15}, which generalised in the Banach space context Lemma 1 in \cite{Lin66} in the same spirit. 

\begin{lemma}
\label{LemmaLindenstraussLip}
Let $M$ be a bounded complete metric space. Let $F\subset M$ be a finite subset of $M$, and let $k\in\mathbb{N}$ and $0<\varepsilon\leq\text{inf}_{p\neq q\in F} d(p,q)$ be given. Then there exists a finite subset $Z\subset M$ with $F\subset Z$ such that for every $\varepsilon$-separated subset $E\subset M$ with $F\subset E$ and $\text{card}(E\setminus F)\leq k$ there is a Lipschitz map $T\colon E\rightarrow Z$ with $T(f)=f$ for all $f\in F$ and 
$$ (1-\varepsilon)d(x,y)\leq d(T(x),T(y))\leq (1+\varepsilon)d(x,y) $$
for all $x,y\in E$.
\end{lemma}
\begin{proof}
Write $R=\text{diam}(M)$ and $F=\{f_1,\dots,f_n\}$. We may assume that $\varepsilon<1$. Consider $E\subset M$ an $\varepsilon$-separated subset with $F\subset E$ and $\text{card}(E\setminus F)\leq k$. We can write this set as $E=\{f_1,\dots,f_n,p^E_1,\dots,p^E_{l_E}\}$ with $l_E\leq k$. Consider now the real valued vector:

$$a_E=(d(f_1,p^E_1),\dots,d(f_1,p^E_{l_E}),\dots,d(p^E_{l_E},p^E_1),\dots,d(p^E_{l_E},p^E_{l_E}))\in \mathbb{R}^{(n+l_E)l_E}. $$
Since $M$ has diameter $R<\infty$, the point $a_E$ belongs to $RB_{\ell_\infty^{(n+l_E)l_E}}$. Hence, if we set 

$$C=\bigsqcup_{l=1}^k RB_{\ell_\infty^{(n+l)l}}, $$
that is, the disjoint union of $RB_{\ell_\infty^{(n+l)l}}$ for $l=1,\dots,k$, then for every set $E\subset M$ with $F\subset E$ and $\text{card}(E\setminus F)\leq k$, the vector $a_E$ belongs to $C$. Since we are working with a finite disjoint union, we can endow $C$ with a metric $d_\infty$ such that $C$ is compact. The restriction of this metric to each $RB_{\ell_\infty^{(n+l)l}}$ coincides with the metric given by the supremum norm, and each $RB_{\ell_\infty^{(n+l)l}}$ is separated at least by $\varepsilon$ from its complementary in $C$. 

Since $C$ is compact, the subset 
$$ A_F=\{a_E\in C\colon F\subset E \text{ and card}(E\setminus F)\leq k\}\subset C$$
is totally bounded in $C$. Hence, given $\varepsilon>0$ there exist $\{E_1,\dots, E_s\}$ with $F\subset E_j$ and $\text{card}(E_j\setminus F)\leq k$ such that $A_F=\bigcup_{j=1}^s B_{\infty}(a_{E_j},\varepsilon^2)$. Set $Z=\bigcup_{j=1}^s E_j$. Let us prove that $Z$ satisfies the thesis of the Lemma.

Clearly, $Z$ is finite and contains $F$. Consider any $\varepsilon$-separated subset $E\subset M$ with $F\subset E$ and $\text{card}(E\setminus F)\leq k$. There exists a $j_0\in \{1,\dots,j\}$ such that $d_\infty(a_{E},a_{E_{j_0}})\leq \varepsilon^2$ and $E_{j_0}\subset F$. Moreover, since $a_E$ and $a_{E_{j_0}}$ are (in particular) closer than $\varepsilon$, they must belong to the same ball $RB_{\ell_{\infty}^{(n+l_0)l_0}}$, so $d_\infty(a_{E},a_{E_{j_0}})=\|a_{E}-a_{E_{j_0}}\|_\infty\leq \varepsilon^2$, and  $\text{card}(E_{j_0})=\text{card}(E)=n+l_0$. Thus, we can write $E=\{f_1,\dots,f_n,p^E_1,\dots,p^E_{l_0}\}$ and $E_{j_0}=\{f_1,\dots,f_n,p^{E_{j_0}}_1,\dots,p^{E_{j_0}}_{l_0}\}$. 

Define now $T\colon E\rightarrow Z$ by $T(f)=f$ if $f\in F$, and $T(p^E_i)=p_i^{E_{j_0}}$ for $i=1,\dots,l_0$. The map $T$ satisfies $T(f)=f$ for all $f\in F$ by definition. We will show that $T$ satisfies the desired inequality for every $x,y\in E$. Since $E$ is the identity on $F$, it is sufficient to check the inequality for pairs of points $x,y\in E$ where $x\notin F$. Then $x=p^E_{i_1}$ for some $1\leq i_1\leq l_0$. If $y=p^E_{i_2}$ for some $1\leq i_2\leq l_0$, then, using the fact that $E$ is $\varepsilon$-separated and the choice of $j_0$, we obtain that
\begin{align*}
    (1-\varepsilon)d(x,y)&\leq d(p_{i_1}^E,p_{i_2}^E)-\varepsilon^2\leq d(p_{i_1}^E,p_{i_2}^E)-\|a_E-a_{E_j}\|_{\infty}\leq d(p^E_{i_1},p^E_{i_2})-(d(p^E_{i_1},p^E_{i_2})-d(p^{E_{j_0}}_{i_1},p^{E_{j_0}}_{i_2}))\\
    &=d(T(x),T(y))\leq \|a_{E_{j_0}}-a_E\|_{\infty}+ d(p^E_{i_1},p^E_{i_2})\leq \varepsilon\varepsilon+d(p^E_{i_1},p^E_{i_2})\leq (1+\varepsilon)d(x,y),
\end{align*}
as desired. If $y\in F$, then the inequality is proven similarly. 
\end{proof}

The proof of the following lemma can be found in \cite[Lemma 5.2]{HajQui22}. We write the statement for the convenience of the reader, as it will be used in the proof of Theorem \ref{theo:metricsimyostseparable}.

\begin{lemma}[\cite{HajQui22}]
\label{denseseparatedsubsets}
Let $M$ be a complete metric space and $(F_n)_{n=1}^\infty$ be a sequence of finite subsets of $M$ with $F_n\subset F_{n+1}$ for all $n\in\mathbb{N}$, and let $(\varepsilon_n)_{n=1}^\infty$ be a decreasing sequence of positive real numbers such that $\varepsilon_n<\inf_{p\neq q\in F_n}d(p,q)$. Then there exists a sequence of sets $(D_n)_{n=1}^\infty$ with the following properties:
\begin{itemize}
    \item[(i)] $D_n\subset D_{n+1}$ for all $n\in\mathbb{N}$,
    \item[(ii)] $F_n\subset D_n$ for all $n\in\mathbb{N}$,
    \item[(iii)] $D_n\cup F_{n+k}$ is $\varepsilon_{n+k}$-separated for all $n\in\mathbb{N}$ and $k\geq 0$,
    \item[(iv)] $D=\bigcup_{n\in\mathbb{N}}D_n$ is dense in $M$,
\end{itemize}
\end{lemma}

With the aid of the previous lemmata we get the following result.

\begin{theorem}\label{theo:metricsimyostseparable}
    Let $M$ be a metric space and let $N$ be a separable subspace of $M$. Then there exists a separable ai-local retract $S$ in $M$ with $N\subseteq S$.
\end{theorem}

\begin{proof}

Let $(p_n)_{n=1}^\infty$ be a dense sequence in $N$. For $n=0$, put $S_0=\{0\}$. Inductively, suppose we have defined $S_{n-1}$, which is finite. Put $F_n=S_{n-1}\cup\{p_n\}$ as a finite set, $\theta_{n}=\inf_{p\neq q\in F_{n}} d(p,q)$ as the separation of said set, and $r_n=\text{rad}(F_n)$ its radius. Set $\varepsilon_n=\min\{1/n,\theta_{n}\}$ and $R_n=\max\{r_n,n\}$. We choose $S_n$ to be the set $Z$ given by Lemma \ref{LemmaLindenstraussLip} applied to $M\cap B(0,R_n)$, which is bounded, with $F=F_n$, $k=n$ and $\varepsilon=\varepsilon_n$. Set $S=\overline{\bigcup_{n\in\mathbb{N}}S_{n}}$. Then clearly $S$ is separable and contains $N$.

Let $(D_n)_{n=1}^\infty$ be the increasing sequence of sets given by Lemma \ref{denseseparatedsubsets} applied to $(F_n)_{n=1}^\infty$ and $(\varepsilon_n)_{n=1}^\infty$. Notice that if $D=\bigcup_{n\in\mathbb{N}}D_n$, then $D$ is dense in $M$ by the Lemma, and $D\cap S$ is dense in $S$ because $S_n\subset D_{n+1}$ for all $n\in\mathbb{N}$.

Fix $n\in\mathbb{N}$, and define the family of subsets:
$$I_{n}=\{E\subset D_n\cap B(0,R_n)\colon F_n\subset E,~ E\text{ is }\varepsilon_{n}\text{-separated, and } \text{card}(E\setminus F_n)\leq n\}. $$
Note that if $E\subset D_n$, the condition that $E$ is $\varepsilon_n$-separated is redundant, but we state it for clarity. Indeed, now it is clear that if $E\in I_n$, then there exists a Lipschitz map $T\colon E\rightarrow S_{n}$ with $(T)_{|F_{n}}=\text{Id}_{F_{n}}$ and 
$$(1-\varepsilon_n)d(x,y)\leq d(T(x),T(y))\leq (1+\varepsilon_n)d(x,y)$$
holds for every $x,y\in E$.

We claim that $\bigcup\limits_{n\in\mathbb N} S_n$ is an ai-local retract in $D$. From here we obtain that $S$ is an ai-local retract in $M$ in virtue of Lemma \ref{lemma:airetractdensigran} and the proof of the theorem would be finished.

In order to prove that hat $\bigcup\limits_{n\in\mathbb N} S_n$ is an ai-local retract in $D$ select $E\subseteq D$ finite and $\varepsilon>0$, and let us find $T:E\rightarrow \bigcup\limits_{n\in\mathbb N}S_n$ such that $(1-\varepsilon)d(x,y)\leq d(T(x),T(y))\leq (1+\varepsilon)d(x,y)$ holds for every $x,y\in E$ and $T(e)=e$ holds for every $e\in E\cap \bigcup\limits_{n\in\mathbb N} S_n$. 

Since the sequence $S_n$ is increasing there exists $m\in\mathbb N$ such that $E\cap \bigcup_{n\in\mathbb N} S_n=E\cap S_k$ holds for every $k\geq m$. Find $n\geq m+1$ such that $E$ is $\varepsilon_n$-separated, $\varepsilon_n<\varepsilon$, $E\subseteq B(0,R_n)$, $E\subseteq D_n$ and $\text{card}(E\setminus F_n)\leq n$ (this can clearly be found since $(\varepsilon_n)\rightarrow 0$ and $(R_n)\rightarrow\infty$).

Consequently, $E\cup F_n\in I_n$. Thus, there exists $T:E\cup F_n\rightarrow S_n$ satisfying that $T(e)=e$ holds for $e\in F_n$ and $(1-\varepsilon_n)d(x,y)\leq d(T(x),T(y))\leq (1+\varepsilon_n)d(x,y)$. Since $\varepsilon_n<\varepsilon$, it follows that, in order to show that $T$ is the desired mapping (up to a composition with the canonical inclusion $S_n\hookrightarrow \bigcup_{k\in\mathbb N} S_k$), we have to prove that $T(e)=e$ holds for every $e\in E\cap \bigcup\limits_{k\in\mathbb N} S_k=E\cap S_m$. By the construction, given $e\in E\cap S_m$, then $e\in F_{m+1}\subseteq F_n$ since $F_n$ is increasing and $n\geq m+1$. Now the fact that $e\in F_n$ and the property defining $T$ gives $T(e)=e$, and the proof is finished. \end{proof}

A more general result can be obtained using transfinite induction argument.

\begin{theorem}\label{theo:metricsimyostnoseparable}
    Let $M$ be a metric space and let $N$ be a subspace of $M$. Then there exists an ai-local retract $S$ in $M$ with $N\subseteq S$ and $dens(S)=dens(N)$.
\end{theorem}

\begin{proof}
Let us prove the result by induction on $\alpha=dens(N)$. If $\alpha=\omega_0$ the statement is simply Theorem \ref{theo:metricsimyostseparable}.

Now assume by inductive step that the result holds true for every cardinal $\beta<\alpha$, and let us prove that the theorem holds true for the cardinal $\alpha$.

So assume that $dens(N)=\alpha$ and take $\{x_\beta: \beta<\alpha\}$ a dense subset of $N$. Call $N_\beta:=\{x_\gamma: \gamma\leq \beta\}$, and it is clear that $N=\overline{\bigcup\limits_{\beta<\alpha} N_\beta}$.

It is clear that $dens(N_\beta)\leq \beta$. By the inductive step, for every $\beta<\alpha$, there exists an ai-local retract $S_\beta$ in $M$ with $dens(S_\beta)=dens(N_\beta)$ and $\bigcup\limits_{\gamma<\beta}S_\gamma \cup N_\beta\subseteq S_\beta$. It is clear that $N\subseteq S:=\overline{\bigcup\limits_{\beta<\alpha}S_\beta}$.

In order to see that $S$ is an ai-local retract of $M$ it is enough in virtue of Lemma \ref{lemma:airetractdensigran} to prove that $\bigcup\limits_{\beta<\alpha} S_\beta$ is an ai-local retract in $M$.

In order to do so, let $E$ be a finite subset of $M$. Since $S_\gamma\subseteq S_\beta$ if $\gamma\leq \beta$ then there exists $\gamma$ large enough so that $E\cap \bigcup\limits_{\beta<\alpha}S_\beta=E\cap S_\gamma$. Since $S_\gamma$ is an ai-local retract in $M$ there exists a map $T:E\rightarrow S_\gamma$ such that $T(e)=e$ for every $e\in E\cap S_\gamma$ and $(1-\varepsilon)d(x,y)\leq d(T(x),T(y))\leq (1+\varepsilon)d(x,y)$ holds for every $x,y\in E$. Consider the inclusion operator $i:S_\gamma\rightarrow \bigcup\limits_{\beta<\alpha} S_\beta$ and $i\circ T: E\rightarrow \bigcup\limits_{\beta<\alpha} S_\beta$ is the desired map.
\end{proof}

We conclude the paper with two brief remarks. First, we provide a second example (this time separable) of an ai-local retract which is not a Lipschitz retract.

\begin{remark}
\label{Remark:Skein}
    Let $M$ be a metric space such that any non-singleton separable subset fails to be a Lipschitz retract (e.g.: the ones constructed in \cite{BVW11} or the space $\text{Sk}(\omega_1)$ in \cite{HajQui23}). If we apply Theorem \ref{theo:metricsimyostseparable} we can find a separable subspace $N\subseteq M$ which is an ai-local retract. Then $N$ is the desired example.

Observe that Theorem \ref{Theorem:LocalRetract_ExtensionintoProper} implies that every local retract in $M$ must be non proper.
\end{remark}

\begin{remark}\label{remark:subesplength}
With the combination of Theorem \ref{theo:metricsimyostnoseparable} and Proposition \ref{prop:lengthlocalretract} we get the following result: Given any complete length metric space $M$ and given any subspace $N\subseteq M$ there exists $S\subseteq M$ which is length such that $N\subseteq S$ and $dens(S)=dens(N)$.   
\end{remark}

\section*{Acknowledgements}

The authors are grateful to Marek C\'uth for fruitful conversations on the topic of the paper. They are also grateful to Prof. Antonio Avil\'es for pointing the authors out about the Urysohn space $\mathbb U$ in Remark \ref{remark:polishurysohn}.

This work was supported by MCIN/AEI/10.13039/501100011033: grant
PID2021-122126NB-C31 (Rueda Zoca), grant PID2021-122126NB-C33 (Quilis)

The research of A. Quilis was also supported by GA23-04776S and project SGS23/056/OHK3/1T/13.

The research of A. Rueda Zoca was also funded by Junta de Andaluc\'ia: Grants FQM-0185 and PY20\_00255; Fundaci\'on S\'eneca: ACyT Regi\'on de Murcia grant 21955/PI/22 and by Generalitat Valenciana project CIGE/2022/97.

\end{document}